\documentclass[12pt]{article}

\usepackage{amsmath,amssymb}
\usepackage{tikz}
\usepackage{color}
\newtheorem{theorem}{Theorem}[section]
\newtheorem{lemma}[theorem]{Lemma}

\newtheorem{proposition}[theorem]{Proposition}

\newcommand{\proof}{\noindent{\bf Proof.\ }}
\newcommand{\qed}{\hfill $\square$ \bigskip}

\usepackage[normalem]{ulem}

\def\cp{\,\square\,}
\DeclareMathOperator {\gp} {gp}

\let\deg\relax
\DeclareMathOperator {\deg} {deg}
\DeclareMathOperator {\Pl} {P\ell}

\usepackage{authblk}
\title{The general position avoidance game and hardness of general position games}
\author[1]{Ullas Chandran S. V.\thanks{Email: \texttt{svuc.math@gmail.com}}}
\author[2,3,4]{Sandi Klav\v{z}ar\thanks{Email: \texttt{sandi.klavzar@fmf.uni-lj.si}}}
\author[1]{Neethu P. K.\thanks{Email: \texttt{p.kneethu.pk@gmail.com}}}
\author[5]{Rudini Sampaio \thanks{Email: \texttt{rudini@dc.ufc.br}}}
\affil[1]{Dept. Mathematics, Mahatma Gandhi College, University of Kerala, Thiruvananthapuram, India}
\affil[2]{Faculty of Mathematics and Physics, University of Ljubljana, Slovenia}
\affil[3]{Faculty of Natural Sciences and Mathematics, University of Maribor, Slovenia}
\affil[4]{Institute of Mathematics, Physics and Mechanics, Ljubljana, Slovenia}
\affil[5]{Dept. Computa\c c\~ao, Universidade Federal do Cear\'a, Fortaleza, Brazil}

\begin{document}
\maketitle
\begin{abstract}
Given a graph $G$, a set $S$ of vertices in $G$ is a general position set if no triple of vertices from $S$ lie on a common shortest path in $G$. The general position achievement/avoidance game is played on a graph $G$ by players A and B who alternately select vertices of $G$. A selection of a vertex by a player is a legal move if it has not been selected before and the set of selected vertices so far forms a general position set of $G$. The player who picks the last vertex is the winner in the general position achievement game and is the loser in the avoidance game. In this paper, we prove that the general position achievement/avoidance games are PSPACE-complete even on graphs with diameter at most 4. For this, we prove that the \textit{mis\`ere} play of the classical Node Kayles game is also PSPACE-complete. As positive results, we obtain linear time algorithms to decide the winning player of the general position avoidance game in rook's graphs, grids, cylinders, and lexicographic products with complete second factors. 
\end{abstract}


\medskip\noindent

{\bf AMS Subj.\ Class:} 05C57; 05C69; 68Q25

\section{Introduction}
\label{sec:intro}
In 1985, Buckley and Harary \cite{buckley-1987} introduced two geodetic games for graphs: the geodetic achievement game and the geodetic avoidance game. In both games, two players (A and B) alternately select vertices which are not in the geodetic closure of the vertices selected so far, where the geodetic closure of a subset $S$ of vertices is the set of all vertices on geodesics (shortest paths) between two vertices in $S$. The first player unable to move is the loser in the achievement game, but is the winner in the avoidance game. 
The authors determine the winner of both games in several families of graphs \cite{buckley-1987}.
In 1988, the result of \cite{buckley-1987} regarding wheel graphs was improved in \cite{nec-1988} and, in 2003, Haynes, Henning and Tiller~\cite{haynes-2003} obtained results for trees and complete multipartite graphs.

Clearly, at the end of both geodetic games, the geodetic closure of the set $S$ of selected vertices is the whole vertex set of the graph, since no other vertex can be selected. However, possibly $S$ is not minimal; that is, $S$ may contain a proper subset whose geodetic closure is the same as $S$.
For example, in a path $P_n$ $v_1v_2\ldots v_n$, a possible sequence of moves is $v_1,\ldots,v_n$ with Player A selecting the vertices $v_i$ for $i$ odd and Player B selecting the vertices $v_i$ for $i$ even. So all vertices may be selected, but the geodetic closure of only two vertices $v_1$ and $v_n$ contains all vertices.

In order to avoid this situation and obtain minimal subsets regarding the geodetic closure, Klav\v{z}ar,  Neethu and Chandran~\cite{klavzar-2021e} introduced two games in 2021: the general position achievement game and the general position avoidance game ($\gp$-{\it achievement} and $\gp$-{\it avoidance} for short).
In both games, the set $S$ of vertices selected by the two players must be a \emph{general position set}, which is a set such that no three distinct vertices from $S$ lie on a common shortest path in $G$.
In the same example of the path $P_n$ $v_1v_2\ldots v_n$ for $n\geq 3$, Player A chooses the first vertex (say $v_j$), Player B chooses the second vertex (say $v_k$ for $k\neq j$), and then Player A is unable to select the third vertex, losing the $\gp$-achievement game and winning the $\gp$-avoidance game.

The {\sc general position problem} of finding a largest general position set of a graph is a generalization of the No-three-in-line problem in the $n\times n$ grid from discrete geometry, which can be traced to the famous Dudeney's ``Puzzle with Pawns'' of his book ``Amusements in Mathematics'' \cite{dudeney-1917} from 1917. In 1995, Korner~\cite{korner-1995} investigated the general position problem on hypercubes, while in~\cite{ullas-2016} it was considered for the first time on general graphs. However, the formalisation of the problem as we know it today and the notation that is in use have been introduced in~\cite{manuel-2018a, manuel-2018b}. Also see~\cite{froese-2017, ku-2018, misiak-2016, payne-2013} for the related general position subset selection problem in computational geometry.
In 2018, it was proved that the {\sc general position problem} is NP-hard \cite{manuel-2018a}.
In 2019, general position sets in graphs were characterized \cite{bijo-2019} and, after this, several additional papers on the general position problem were published, many of them with bounds on the maximum size of a general position set and exact values in graph products \cite{thomas-2020, klavzar-2021d, klavzar-2021a, klavzar-2021b, klavzar-2019, patkos-2020,  tian-2021a,tian-2021b}.

As mentioned before, the $\gp$-achievement game and the $\gp$-avoidance game were introduced in 2021 in~\cite{klavzar-2021e}.
In the game terminology, the $\gp$-achievement game is the normal game (the last to play wins) and the $\gp$-avoidance game is the mis\`{e}re game (the last to play loses).
From the classical Zermelo-von Neumann theorem, in both games one of the two players has a winning strategy, since they are finite perfect-information games without draw ~\cite{zermelo13}.
So, the main question is: given a graph in one of the two general position games, which player has a winning strategy?
The two games are independent, that is, winning the $\gp$-achievement game in a graph does not mean that the player loses the $\gp$-avoidance game in the same graph. 
Figure~\ref{fig1} shows a graph on which player A has a winning strategy in both games. In both games, A first selects the central vertex. By the symmetry of the graph, we may assume w.l.o.g. that B replies by playing the vertex $2$. With this, optimal strategies for A are shown in the figure.

\begin{figure}[ht!]\centering\scalebox{1.0}{
	\begin{tikzpicture}
	\tikzstyle{vertex}=[draw,circle,fill=white!25,minimum size=14pt,inner sep=0.1pt]
	\tikzstyle{vertex2}=[draw,circle,fill=black!25,minimum size=14pt,inner sep=0.1pt]

	\node[vertex](u) at (4,1) {1};
	\node[vertex](c1) at (2,2) {2};
	\node[vertex](c2) at (4,2) {};
	\node[vertex](c3) at (6,2) {};

	\node[vertex](s1) at (2,0) {3};
	\node[vertex](s2) at (4,0) {};
	\node[vertex](s3) at (6,0) {};

	\draw (c1)--(c2)(c2)--(c3);
	\draw (s1)--(s2)(s2)--(s3);
	\draw (u)--(c1)(u)--(c2)(u)--(c3);
	\draw (u)--(s1)(u)--(s2)(u)--(s3);
    \path[-](c1) edge[out=300, in=60] (s1)
            (c2) edge[out=330, in=30] (s2)
            (c3) edge[out=300, in=60] (s3)
            (c1) edge[out=30, in=150] (c3)
            (s1) edge[out=-30, in=210] (s3);

	\node[vertex] (u) at (4+7,1) {1};
	\node[vertex](c1) at (2+7,2) {2};
	\node[vertex](c2) at (4+7,2) {3};
	\node[vertex](c3) at (6+7,2) {4};

	\node[vertex](s1) at (2+7,0) {};
	\node[vertex](s2) at (4+7,0) {};
	\node[vertex](s3) at (6+7,0) {};

	\draw (c1)--(c2)(c2)--(c3);
	\draw (s1)--(s2)(s2)--(s3);
	\draw (u)--(c1)(u)--(c2)(u)--(c3);
	\draw (u)--(s1)(u)--(s2)(u)--(s3);
    \path[-](c1) edge[out=300, in=60] (s1)
            (c2) edge[out=330, in=30] (s2)
            (c3) edge[out=300, in=60] (s3)
            (c1) edge[out=30, in=150] (c3)
            (s1) edge[out=-30, in=210] (s3);

	\end{tikzpicture}}
	\caption{\label{fig1}Player A wins in both games. The sequence of the chosen vertices is represented by numbers. Odd numbers are player A moves. The game $\gp$-achievement is illustrated in the left and the game $\gp$-avoidance in the right.}
\end{figure}
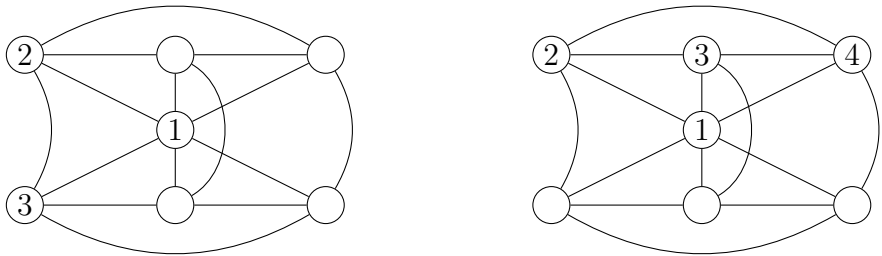

In \cite{klavzar-2021e}, the $\gp$-achievement game was investigated on Hamming graphs and Cartesian and lexicographic products. Among other results, it was proved that player A wins the $\gp$-achievement game on a bipartite graph $G$ if and only if the number of isolated vertices in $G$ is odd.

In this paper, we determine the computational complexity of both games: the $\gp$-achievement game and the $\gp$-avoidance game are PSPACE-complete even in graphs with diameter at most $4$. We obtain the hardness of the $\gp$-achievement game by a reduction from the classical Node Kayles game, which was proved to be PSPACE-complete~\cite{schaefer-1978} in 1978. The hardness of the $\gp$-avoidance game is obtained by a reduction from the \textit{mis\`ere} Node Kayles game, which is just as the normal Node Kayles game, except that the last player to move loses the game. In order to obtain our result, we first prove that the \textit{mis\`ere} Node Kayles game is also PSPACE-complete.
To the best of our knowledge, although the normal Node Kayles game was proved PSPACE-complete in 1978, the PSPACE-hardness of the \textit{mis\`ere} Node Kayles game was not proved before.

In the second part of the paper we complement results from~\cite{klavzar-2021e} regarding the $\gp$-achievement game by proving several results for the $\gp$-avoidance game on Cartesian and lexicographical products. We obtain linear time algorithms to decide the winning player of the general position avoidance game in rook's graphs, grids, cylinders, and lexicographic products with complete second factors.

\section{Preliminary results and examples}
\label{sec:preliminary}

In this section we first recall other concepts and  notations needed. 

All graphs considered, here, are finite, simple and without loops or multiple edges. The {\em distance} $d_G(u,v)$ between vertices $u$ and $v$ is the length of a shortest $u,v$-path. A $u,v$-path of minimum length is also called a $u,v$-{\it geodesic}. The {\em interval} $I_G[u,v]$ between $u$ and $v$  is the set of vertices that lie on some $u,v$-geodesic. For $S\subseteq V(G)$, set $I_G[S]=\bigcup_{_{u,v\in S}}I_G[u,v]$.
When the graph $G$ is clear from the context, we omit the subscript and simply write $d(u,v)$, $I[u,v]$ and $I[S]$.

A set of vertices $S\subseteq V(G)$ is a {\em general position set} if no three vertices from $S$ lie on a common geodesic. The size of a smallest general position set of $G$ is denoted by $\gp(G)$ and called the {\em general position number} of $G$.
For example, notice that $\gp(C_n)=3$ for any cycle $C_n$ with $n\geq 5$.

Following the notations from~\cite{klavzar-2021e}, the sequence of vertices played in the position games on a graph $G$ will be denoted by $a_1, b_1, a_2, b_2, \ldots$, that is, the vertices played by A are $a_1, a_2, \ldots$, and the vertices played by B are $b_1, b_2, \ldots$ For instance, we may say that A starts the game by playing $a_1 = x$, where $x\in V(G)$. Suppose that $x_1, \ldots, x_j$ are vertices played so far on the graph $G$. Then we say that $y\in V(G)$ is a {\em playable vertex} if $y\notin \{x_1, \ldots, x_j\}$ and $\{x_1, \ldots, x_j\} \cup \{y\}$ is a general position set of $G$. Let $\Pl_G(x_1, \ldots, x_j)$ be the set of all playable vertices after the vertices $x_1, \ldots, x_j$ have already been played; we may sometimes simplify the notation $\Pl_G(x_1, \ldots, x_j)$ to $\Pl_G(\ldots x_j)$. For instance, if $x$ and $y$ are arbitrary vertices of a path $P$, then $\Pl_P(x) = V(P)\setminus \{x\}$ and $\Pl_P(x,y) = \emptyset$. In fact, if $G$ is an arbitrary graph, then $\Pl_G(a_1) = V(G)\setminus \{a_1\}$. Denoting by $S$ the set of vertices $\{x_1, \ldots, x_j\}$  played so far, we may also write $\Pl_G(S)$ for $\Pl_G(x_1, \ldots, x_j)$.  The following observation will turn out to be quite useful, hence we state it as a lemma for further use.

\begin{lemma}
\label{lem:playable}{\rm \cite{klavzar-2021e}}
Let $S$ be the sequence of played vertices so far on a graph $G$. Then $x\in \Pl_G(S)$ if and only if the following two conditions hold:  
\begin{enumerate}
\item[(i)] if $u,v\in S$, then $x\notin I[u,v]$, and 
\item[(ii)] if $u\in S$, then $I[x,u]\cap S = \{u\}$. 
\end{enumerate}
\end{lemma}

We next recall a characterization of general position sets from~\cite{bijo-2019}, for which some preparation is needed. Let $G$ be a connected graph, $S\subseteq V(G)$, and ${\cal P} = \{S_1, \ldots, S_p\}$ a partition of $S$. Then ${\cal P}$ is \emph{distance-constant} if for any $i,j\in [p]$, $i\ne j$, the value $d(u,v)$, where $u\in S_i$ and $v\in S_j$ is independent of the selection of $u$ and $v$.  If ${\cal P}$ is a distance-constant partition, and $i,j\in [p]$, $i\ne j$, then let $d(S_i, S_j)$ be the distance between the sets $S_i$ and $S_j$, that is, the distance between one arbitrary vertex of $S_i$ and one of $S_j$. We say that a distance-constant partition ${\cal P}$ is {\em in-transitive} if $d(S_i, S_k) \ne d(S_i, S_j) + d(S_j,S_k)$ holds for arbitrary pairwise different $i,j,k\in [p]$. With this notations in hand, we have the following characterization.
\begin{theorem}
{\rm\cite[Theorem 3.1]{bijo-2019}}
\label{thm:2.1}
Let $G$ be a connected graph. Then $S\subseteq V(G)$ is a general position set if and only if the components of $G[S]$ are complete subgraphs, the vertices of which form an in-transitive, distance-constant partition of $S$.
\end{theorem}

We will also make use of the following fact that was observed for the first time in the proof of~\cite[Theorem 5.1]{bijo-2019}. 

\begin{lemma}
{\rm\cite{bijo-2019}}
\label{lem:bipartite}
Let $G$ be a connected bipartite graph. If $S$ is a general position set of $G$ with $|S|\geq 3$, then $S$ is an independent set of $G$. In other words, if $S$ is not an independent set of $G$, then $S$ is a general position set if and only if $S$ consists of exactly two adjacent vertices.
\end{lemma} 

Let us first look at some examples. Since in a complete graph every vertex subset is a general position set, A wins the $\gp$-avoidance game on the complete graph $K_n$(and loses the $\gp$-achievement game) if and only if $n$ is even. Clearly A wins the $\gp$-avoidance game on  a graph $G$ with $\gp(G) = 2$. As proved in~\cite{ullas-2016}, the only graphs with  $\gp(G) = 2$ are paths and the cycle $C_4$. If $\gp(G) = 3$, then $\gp$-avoidance game will take either two or three moves. In fact, if $\gp(G) = 3$ then A wins the $\gp$-avoidance game if only only if every vertex of $G$ lies in a maximal general position set of order $3$. Applying this observation to cycles we infer that B wins the $\gp$-avoidance game on the cycle $C_{n}(n\ge 3)$ if and only if $n\neq 4$. On the other hand, it was observed in~\cite{klavzar-2021e} that B wins the $\gp$-achievement game on $C_n$ if and only if $n$ is even. Hence both the games are intrinsically different.

The following useful result for the $\gp$-avoidance game is a direct consequence of \cite[Theorem 3.1]{klavzar-2021e} which deals with the $\gp$-achievement game. 

\begin{theorem}
\label{thm:up-to-gp-set} 
Let $G$ be a graph. Then the following holds. 

(i) If A has a strategy such that after the vertex $a_k$, for some $k\ge 1$, is played, the set $\Pl_G(\ldots a_k)\cup \{a_1, b_1, \ldots, a_k\}$ is a general position set and $|\Pl_G(\ldots a_k)|$ is odd, then A wins the $\gp$-avoidance game. 

(ii) If B has a strategy such that after the vertex $b_k$, for some $k\ge 1$, is played, the set $\Pl_G(\ldots b_k)\cup \{a_1, b_1, \ldots, b_k\}$ is a general position set and $|\Pl_G(\ldots b_k)|$ is odd, then B wins the $\gp$-avoidance game. 
\end{theorem}

As an example of the application of Theorem~\ref{thm:up-to-gp-set}, consider the Petersen graph $P$ (see Figure \ref{fig1b}).
By the symmetry of $P$, we may assume that $a_1$ is an arbitrary vertex of $P$.
Consider first the case that $b_1$ is a neighbor of $a_1$, as illustrated in Figure \ref{fig1b}. Then it is easy to check that $\Pl_P(a_1,b_1)$ consists of four vertices. Moreover, $\Pl_P(a_1,b_1)\cup \{a_1, b_1\}$ is a general position set. Hence, after A plays any vertex $a_2$ from $\Pl_P(a_1,b_1)$, we have that $\Pl_P(a_1,b_1,a_2)\cup\{a_1,b_1,a_2\}$ is a general position set, since it is equal to $\Pl_P(a_1,b_1)\cup\{a_1,b_1\}$, and $|\Pl_P(a_1,b_1,a_2)|=3$ is odd. Thus, we can apply Theorem~\ref{thm:up-to-gp-set}(i) to conclude that A wins the $\gp$-avoidance game on $P$. Consider second the case that $b_1$ is a vertex at distance $2$ from $a_1$. Then A replies with $a_2$ which is a neighbor of $a_1$ but not a neighbor of $b_1$. Now it is easy to check that $\Pl_P(a_1,b_1,a_2)$ consists of three vertices and $\Pl_P(a_1,b_1,a_2)\cup\{a_1,b_1,a_2\}$ is a general position set, hence Theorem~\ref{thm:up-to-gp-set}(i) can be applied once more. (We remark in passing that B wins the $\gp$-achievement game on the Petersen graph.)

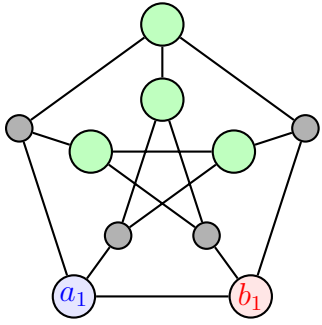
\begin{figure}[ht!]\centering\scalebox{1.0}{
\begin{tikzpicture}[style=thick]
\tikzstyle{vertex1}=[draw,circle,fill=black!30,minimum size=10pt,inner sep=0.1pt]
\tikzstyle{vertexA}=[draw,circle,fill=blue!10,minimum size=16pt,inner sep=0.1pt]
\tikzstyle{vertexB}=[draw,circle,fill=red!10,minimum size=16pt,inner sep=0.1pt]
\tikzstyle{vertex2}=[draw,circle,fill=green!25,minimum size=16pt,inner sep=0.1pt]
\node[vertex2](u1) at ( 18:1cm) {};
\node[vertex2](u2) at ( 90:1cm) {};
\node[vertex2](u3) at (162:1cm) {};
\node[vertex1](u4) at (234:1cm) {};
\node[vertex1](u5) at (306:1cm) {};
\node[vertex1](v1) at ( 18:2cm) {};
\node[vertex2](v2) at ( 90:2cm) {};
\node[vertex1](v3) at (162:2cm) {};
\node[vertexA](v4) at (234:2cm) {\textcolor{blue}{$a_1$}};
\node[vertexB](v5) at (306:2cm) {\textcolor{red}{$b_1$}};

\draw (u1)--(u4)--(u2)--(u5)--(u3)--(u1);
\draw (v1)--(v2)--(v3)--(v4)--(v5)--(v1);
\draw (u1)--(v1);
\draw (u2)--(v2);
\draw (u3)--(v3);
\draw (u4)--(v4);
\draw (u5)--(v5);
\end{tikzpicture}}
\caption{\label{fig1b}Example of Theorem \ref{thm:up-to-gp-set}(i) on the Petersen graph $P$, when Bob selects a neighbor $b_1$ of $a_1$. The set $\Pl_P(a_1,b_1)$ consists of the green vertices.}
\end{figure}

As another application of Theorem~\ref{thm:up-to-gp-set} we have the following result. 

\begin{proposition}
\label{prp:complete-bipartite}
Let $G$ be the complete multipartite graph $K_{n_1,\ldots, n_k}$, where $k\ge 2$ and $n_i\ge 2$ for $i\in [k]$. Then A wins the $\gp$-avoidance game on $G$ if and only if $k$ is even and at least one $n_i$ is even.  
\end{proposition}

\proof Suppose first that $n_i$ is odd for all $i\in [k]$. Let $X$ be the partition set of $G$ in which the first move $a_1$ has been played by A. Then B replies by playing a vertex $b_1\ne a_1$ from $X$. Note that $\Pl_G(a_1, b_1) = X\setminus \{a_1, b_1\}$. Since $X$ is a general position set of $G$, Theorem~\ref{thm:up-to-gp-set}(ii) applies and B wins the $\gp$-avoidance game on $G$. 

Hence if at least one $n_i$ is even, A would definitely choose $a_1$ from that even partition set $X$ to win the game. Now, if B would reply by playing a vertex in $X$, then by the argument of the previous paragraph and with Theorem~\ref{thm:up-to-gp-set}(i) in hand, A would win. So it is better for B to play a vertex $b_1$ which lies in a partition set $Y\ne X$. Since $\Pl_G(a_1,b_1) = V(G)\setminus (X\cup Y)$,  the vertex $a_2$ must lie in a partition set $Z$ different from both $X$ and $Y$. Continuing in this manner, each of the subsequent played vertices belongs to its private partition set. In conclusion, if some $n_i$ is even, then A will win the $\gp$-avoidance game if and only if $k$ is even.
\qed

The {\it generalized wheel} $W_{n,m}$, $n\ge 1$, $m\ge 3$, is the graph obtained by joining
each vertex of the complement of the complete graph $K_n$ to every vertex of the cycle $C_m$. That is, $W_{n,m}$ is the join of $\overline{K_n}$ and $C_m$. In the following we resolve the $\gp$-avoidance game on $W_{n,m}$.

\begin{theorem}
\label{thm:4.2}
If $n\ge 1$ and $m\ge 3$, then B wins the $\gp$-avoidance game on $W_{n,m}$ if and only if $m\ge 4$. 
\end{theorem}

\proof 
Suppose that $m=3$. Let $a_1$ be an arbitrary vertex of $C_3$. Assume first that B replies by playing an arbitrary vertex of $\overline{K_n}$. Then $\Pl_G(a_1,b_1) = V(C_3)\setminus \{a_1\}$ and $\Pl_G(a_1,b_1) \cup \{a_1, b_1\}$ is isomorphic to $K_4$ which is a general position set. Then A chooses a vertex of $\overline{K_n}$ leaving only one option to Bob (the last non-selected vertex of $\overline{K_n}$) and A wins. Assume second that $b_1$ is a vertex of $C_3$. Then A replies by playing a vertex of $\overline{K_n}$ and we are in the same situation as in the first case. Hence in any case A wins.  

Suppose now that $m\ge 4$. If $A$ starts with a vertex from $C_m$, then choose $b_1$ from $\overline{K_n}$. Thus by Theorem \ref{thm:2.1}, the remaining game is restricted to a clique in $W_{n,m}$. As $m\ge 4$, each maximal clique of $W_{n,m}$ has order $3$, hence B wins the game. And if $A$ starts with a vertex of $K_m$, then by choosing a vertex from $C_m$, B again wins the game.
\qed

This is in contrast with the fact that solving $\gp$-achievement game on $W_{n,m}$ appears to be difficult. This is because, in the $\gp$-achievement game, if $n$ is even, A plays $a_1$ in $C_m$ (otherwise $B$ wins by playing only in $\overline{K_n}$) and consequently B also plays $b_1$ in $C_m$ (otherwise A wins immediately by playing in $C_m$) and no other vertex of $\overline{K_m}$ can be selected. In this case, the $\gp$-achievement game on $C_m$ reduces to the game of selecting vertices in such a way that no three selected vertices induce a $P_3$, which seems to be hard even in cycles.

\section{Both games are PSPACE-complete}
\label{sec:complexity}

In this section, we prove that the general position achievement and avoidance games are PSPACE-complete. We first show that they are in PSPACE.
Here we consider the games as decision problems: given a graph, does player A have a winning strategy?

\begin{lemma}\label{lem-pspace}
The $\gp$-achievement and $\gp$-avoidance games are in PSPACE.
\end{lemma}

\begin{proof}
Since the number of turns is at most~$n=|V(G)|$ and, in each turn, the number of possible moves is at most~$n$ (there are at most~$n$ vertices to select), we have that the $\gp$-achievement game and the $\gp$-avoidance game are polynomially bounded two player games, which implies that they are in PSPACE~\cite{demaine-2009}.
\qed
\end{proof}

Now we obtain reductions from the clique-forming game, which is PSPACE-complete \cite{schaefer-1978}. In this game, given a graph $G$, two players 1 and 2 select alternately vertices and the subset of the chosen vertices must induce a clique. The first player unable to play loses the game. That is, the player who has played the last vertex of a maximal clique wins.  
This game is strongly related to the classical Node Kayles game, that is also PSPACE-complete \cite{schaefer-1978}, in which the objective is to obtain an independent set, instead of a clique.
The clique-forming game is the Node Kayles game played on the complement of the graph, and vice-versa.

\begin{theorem}\label{teo-pspace1}
The $\gp$-achievement game is PSPACE-complete even on graphs with diameter at most $4$.
\end{theorem}

\begin{proof}
From Lemma \ref{lem-pspace}, the $\gp$-achievement game is in PSPACE.
Now we obtain a reduction from the clique-forming game. Let $H$ be an instance of clique-forming game with vertex set $V(H)=\{v_1,\ldots,v_n\}$. We will construct a graph $G$ such that player A has a winning strategy in the $\gp$-achievement game on $G$ if and only if the second player of the clique-forming game on $H$ has a winning strategy.

Let $G$ be the graph obtained from $H$ by adding a new vertex $u$ adjacent to all vertices of $H$ and a new vertex $f_i$ (called ``friend'' of $v_i$), for every vertex $v_i$ of $H$, whose only neighbor is $v_i$.
Notice that $G$ has diameter at most 4. See Figure \ref{fig2} for an example.

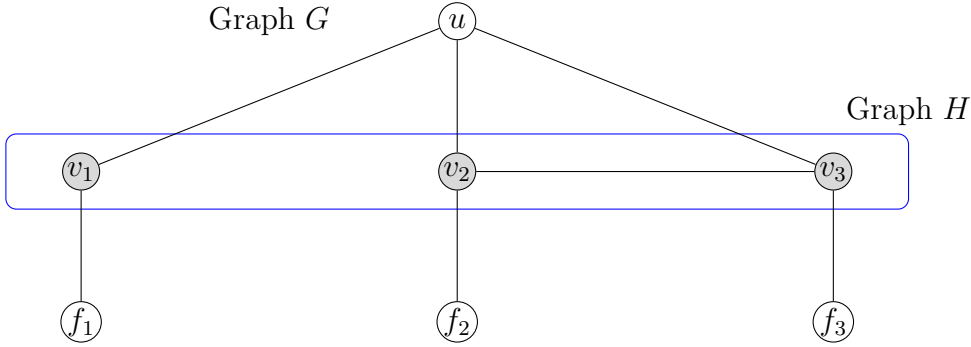
\begin{figure}[ht!]\centering\scalebox{1.0}{
\begin{tikzpicture}
\tikzstyle{vertex}=[draw,circle,fill=white!25,minimum size=14pt,inner sep=.1pt]
\tikzstyle{vertex2}=[draw,circle,fill=black!15,minimum size=14pt,inner sep=0.1pt]

\draw[rounded corners,blue] (1.0,1.5) rectangle (13,2.5);
\node at (13,2.8) {Graph $H$};
\node at (4.5,4) {Graph $G$};

\node[vertex](u) at (7,4) {$u$};

\node[vertex2](v1) at (2,2) {$v_1$};
\node[vertex2](v2) at (7,2) {$v_2$};
\node[vertex2](v3) at(12,2) {$v_3$};

\node[vertex](f1) at (2,0) {$f_1$};
\node[vertex](f2) at (7,0) {$f_2$};
\node[vertex](f3) at(12,0) {$f_3$};

\draw (u)--(v1)(u)--(v2)(u)--(v3);
\draw (v1)--(f1)(v2)--(f2)(v3)--(f3);
\draw (v2)--(v3);

\end{tikzpicture}}
\caption{\label{fig2}Example of the reduction from the clique-forming game to the $\gp$-achievement game, where $H$ is the graph with vertices $v_1$, $v_2$ and $v_3$. Player A has a winning strategy in the $\gp$-achievement game on $G$ if and only if Player 2 has a winning strategy in the clique-forming game on $H$.}
\end{figure}

If player A selects a vertex $v_i$ of $H$ in the first move of the $\gp$-achievement game played on $G$, then player B wins immediately by selecting the vertex $f_i$: notice that the shortest path between $f_i$ and a vertex other than $v_i$ passes through $v_i$. From the same argument, if player A selects a vertex $f_i$, where $v_i\in V(H)$, in the first move, then player B wins immediately by selecting the vertex $v_i$. So let us assume that player A selects the vertex $u$ in the first move.

From now on, notice that players A and B cannot select two non-adjacent vertices $v_i$ and $v_j$ of $H$, since $v_i-u-v_j$ is a shortest path. Moreover, during the $\gp$-achievement game, they cannot select two vertices $f_i$ and $f_j$ such that $v_i$ and $v_j$ are non-adjacent, since $f_i-v_i-u-v_j-f_j$ is a shortest path. They cannot select two vertices $v_i$ and $f_j$ such that $v_i$ and $v_j$ are non-adjacent, since $v_i-u-v_j-f_j$ is a shortest path.
Finally, the players cannot select two vertices $v_i$ and $f_i$, since $u-v_i-f_i$ is a shortest path.

Let $C$ be the vertex subset obtained from the selected vertices after the $\gp$-achievement game by the following: for every played vertex $v_i$ of $H$, put $v_i$ on $C$ and, for every played vertex $f_i$, put $v_i$ on $C$.
From the above, we have that $C$ is a clique of $H$ and we may assume that the players only select vertices from $H$ after the first move (on $u$), since playing on the $f_i$'s is essentially the same as playing on the corresponding vertices of $H$.

Thus, if the second player of the clique-forming game on $H$ has a winning strategy,  player A has a winning strategy in the $\gp$-achievement game on $G$, since player A will be the second to play on $H$ (recall that $u$ was the first chosen vertex). Moreover, if the first player of the clique-forming game on $H$ has a winning strategy, player B has a winning strategy in the $\gp$-achievement game on $G$, since player B will be the first to play on $H$.
\qed
\end{proof}

In the following, we prove that deciding who wins the $\gp$-avoidance game is PSPACE-complete. We obtain a reduction from the \textit{mis\`ere} clique-forming game, which is the same as the clique-forming game, but the first player unable to play wins the game.
As before, the \textit{mis\`ere} clique-forming game is strongly related to the \textit{mis\`ere} Node Kayles game, whose objective is to obtain an independent set and the first player unable to play wins the game.
Although these games are well known and their normal versions were proved PSPACE-complete in 1978, to the best of our knowledge the hardness of these \textit{mis\`ere} games were not proved before.
In this paper, we prove that both \textit{mis\`ere} games are PSPACE-complete.
We postpone the proof of this to Theorem \ref{teo-pspace3}.
In the next theorem, we use this result to prove that the $\gp$-avoidance game is PSPACE-complete.

\begin{theorem}\label{teo-pspace2}
The $\gp$-avoidance game is PSPACE-complete even in graphs with diameter at most $4$.
\end{theorem}

\begin{proof}
From Lemma \ref{lem-pspace}, the $\gp$-avoidance game is in PSPACE.
Now we obtain a reduction from the \textit{mis\`ere} clique-forming game, the decision version of which is proved to be PSPACE-complete in upcoming Theorem \ref{teo-pspace3}. Let $H$ be an instance of the \textit{mis\`ere} clique-forming game with vertex set $V(H)=\{v_1,\ldots,v_n\}$. We will construct a graph $G$ such that player A has a winning strategy in the $\gp$-avoidance game on $G$ if and only if the second player of the \textit{mis\`ere} clique-forming game on $H$ has a winning strategy.

Let $G$ be the graph obtained from $H$ by adding a new vertex $u$ adjacent to all vertices of $H$ and, for every vertex $v_i$ of $H$, by adding a $C_5$ with vertices $F_i=\{p_i,q_i,r_i,s_i,t_i\}$, all of them adjacent to $v_i$. Notice that $G$ has diameter at most 4. See Figure \ref{fig3} for an example.

\begin{figure}[ht!]\centering\scalebox{1.0}{
\begin{tikzpicture}
\tikzstyle{vertex}=[draw,circle,fill=white!25,minimum size=14pt,inner sep=0.1pt]
\tikzstyle{vertex2}=[draw,circle,fill=black!15,minimum size=14pt,inner sep=0.1pt]

\draw[rounded corners,blue] (1.0,1.5) rectangle (13,2.5);
\node at (13,2.8) {Graph $H$};
\node at (4.5,4) {Graph $G$};

\node[vertex](u) at (7,4) {$u$};

\node[vertex2](v1) at (2,2) {$v_1$};
\node[vertex2](v2) at (7,2) {$v_2$};
\node[vertex2](v3) at(12,2) {$v_3$};

\node[vertex](a1) at (0,0) {$p_1$};
\node[vertex](b1) at (1,0) {$q_1$};
\node[vertex](c1) at (2,0) {$r_1$};
\node[vertex](d1) at (3,0) {$s_1$};
\node[vertex](e1) at (4,0) {$t_1$};

\node[vertex](a2) at (5,0) {$p_2$};
\node[vertex](b2) at (6,0) {$q_2$};
\node[vertex](c2) at (7,0) {$r_2$};
\node[vertex](d2) at (8,0) {$s_2$};
\node[vertex](e2) at (9,0) {$t_2$};

\node[vertex](a3) at(10,0) {$p_3$};
\node[vertex](b3) at(11,0) {$q_3$};
\node[vertex](c3) at(12,0) {$r_3$};
\node[vertex](d3) at(13,0) {$s_3$};
\node[vertex](e3) at(14,0) {$t_3$};

\draw (u)--(v1)(u)--(v2)(u)--(v3);
\draw (v1)--(a1)(v2)--(a2)(v3)--(a3);
\draw (v1)--(b1)(v2)--(b2)(v3)--(b3);
\draw (v1)--(c1)(v2)--(c2)(v3)--(c3);
\draw (v1)--(d1)(v2)--(d2)(v3)--(d3);
\draw (v1)--(e1)(v2)--(e2)(v3)--(e3);
\draw (v2)--(v3);
\draw (a1)--(b1)--(c1)--(d1)--(e1);
\draw (a2)--(b2)--(c2)--(d2)--(e2);
\draw (a3)--(b3)--(c3)--(d3)--(e3);
\path[-](a1) edge[out=-30, in=210] (e1);
\path[-](a2) edge[out=-30, in=210] (e2);
\path[-](a3) edge[out=-30, in=210] (e3);

\end{tikzpicture}}
\caption{\label{fig3}Example of the reduction from the \textit{mis\`ere} clique-forming game to the $\gp$-avoidance game, where $H$ is the graph with vertices $v_1$, $v_2$ and $v_3$. Player A has a winning strategy in the $\gp$-avoidance game on $G$ if and only if Player 2 has a winning strategy in the clique-forming game on $H$.}
\end{figure}
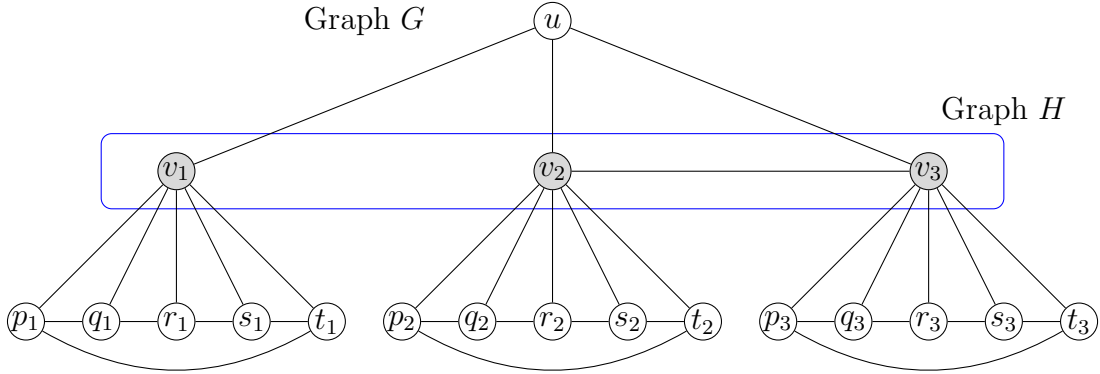

If player A selects a vertex $v_i$ of $H$ in the first move of the $\gp$-avoidance game played on $G$, then player B wins by selecting the vertex $q_i$: notice that, for every vertex $w\not\in\{v_i,p_i,r_i\}$, there is a shortest path between $w$ and $q_i$ passing through $v_i$, and then player A must select $p_i$ or $r_i$, being the last to move and losing the game. From the same argument, if player A selects a vertex of $F_i$, for some $v_i\in V(H)$, in the first move, then player B wins by selecting the vertex $v_i$. So we may assume that player A selects the vertex $u$ in the first move.

From now on, notice that players A and B cannot select two non-adjacent vertices $v_i$ and $v_j$ of $H$, since $v_i-u-v_j$ is a shortest path. Also, during the $\gp$-avoidance game, they cannot select a vertex $f_i\in F_i$ and a vertex $f_j\in F_j$ such that $v_i$ and $v_j$ are non-adjacent, since $f_i-v_i-u-v_j-f_j$ is a shortest path. They cannot select two vertices $v_i$ and $f_j\in F_j$ such that $v_i$ and $v_j$ are non-adjacent, since $v_i-u-v_j-f_j$ is a shortest path. Moreover, the players cannot select two vertices $v_i$ and $f_i\in F_i$, since $u-v_i-f_i$ is a shortest path.

Finally, if exactly one vertex of $F_i$ was already selected by one of the players and player A selects a second vertex of $F_i$, then player B can select a third vertex of $F_i$, forcing player A to select a vertex outside $F_i$ in the next move.
This is because $F_i$ induces the cycle $C_5$ and all maximal general position sets of $C_n$ with $n\geq 5$ are of order 3.
The same occurs if player B selects a second vertex of $F_i$.
For example, if $p_i$ was already selected and player A chooses vertex $q_i$, player B can select the vertex $s_i$.
In other words, playing twice in the same set $F_i$ will not change the outcome of the game.

Let $C$ be the vertex subset obtained from the selected vertices after the $\gp$-avoidance game by the following: for every played vertex $v_i$ of $H$, put $v_i$ on $C$ and, for every played vertex $f_i\in F_i$, put $v_i$ on $C$.
From the above, we have that $C$ is a clique of $H$ and we may assume that the players only select vertices from $H$ after the first move (on $u$), since playing on the sets $F_i$'s is essentially the same as playing on the corresponding vertices of $H$ (recall that playing twice in the same set $F_i$ will not change the outcome of the game).

Thus, if the second player of the \textit{mis\`ere} clique-forming game on $H$ has a winning strategy, player A has a winning strategy in the $\gp$-avoidance game on $G$, since player A will be the second to play on $H$ (recall that $u$ was the first chosen vertex). Moreover, if the first player of the \textit{mis\`ere} clique-forming game on $H$ has a winning strategy, player B has a winning strategy in the $\gp$-avoidance game on $G$, since player B will be the first to play on $H$.
\qed
\end{proof}

Finally, we prove the PSPACE-hardness of the \textit{mis\`ere} clique-forming game and the \textit{mis\`ere} Node Kayles game.
The complexity of computational games defined a long time ago has been the subject of recent papers (see for example \cite{rudini-2020,rudini-2022}).
We first reproduce in Theorem \ref{teo-pspace-1978} the 1978 PSPACE-hardness proof of the normal Node Kayles game \cite{schaefer-1978}, since our proof in Theorem \ref{teo-pspace3} is strongly based on it.

\begin{theorem}\label{teo-pspace-1978}[Schaefer, 1978]
The normal Node Kayles game is PSPACE-complete.
\end{theorem}

\begin{proof}
The reduction is from the TQBF problem, which takes as an instance a totally quantified CNF formula $\Phi$ with $n$ variables $x_1,\ldots,x_n$ and $m$ clauses $B_1,\ldots,B_m$, where the quantifiers are alternately $\exists$ and $\forall$ starting from $x_n$ to $x_1$. 
This can be seen as a game in which Player 1 and Player 2 alternately set true or false to the variables $x_n,x_{n-1},\ldots,x_1$ in this order. Player 1 wins if the formula is true at the the end; otherwise, Player 2 wins.

Consider the example $\Phi=\exists x_3\ \forall x_2\ \exists x_1:\ B_1\wedge B_2\wedge B_3$, with the following clauses: $B_1=(x_1\vee\overline{x_1})$, $B_2=(x_1\vee\overline{x_2})$ and $B_3=(\overline{x_1}\vee x_2\vee x_3)$. In this example, Player 1 wins, since Player 1 sets true to the variable $x_3$ in the first turn (in order to satisfy clause $B_3$), Player 2 then sets true to the variable $x_2$ (trying to make $B_2$ unsatisfied) and Player 1 must set true to variable $x_1$, satisfying all clauses and winning the game.

Schaefer \cite{schaefer-1978} assumes that the number $n$ of variables is odd (otherwise we can add a new variable which does not appear in any clause) and the first clause $B_1$ is $(x_1\vee\overline{x_1})$, which is always satisfied and does not change the outcome of the game. Given a TQBF formula $\Phi$, the constructed graph $G=G(\Phi)$ is described below.

For $i\in [n]$, create in $G$ the vertices $x_i$ and $\overline{x_i}$ associated to the variable $x_i$. For $k\in [m]$, create the vertex $x_{0,k}$ associated to the clause $B_k$. For $i\in [n]$ and $j\in \{0,\ldots,i-1\}$, create the vertex $y_{i,j}$. Let $X_i$ be the set with all the vertices $x_i$, $\overline{x_i}$ and $y_{i,j}$, and make $X_i$ be a clique in $G$. Let $X_0$ be the set with the vertices $x_{0,1},\ldots,x_{0,m}$ associated to the clauses, and make $X_0$ be a clique in $G$.

If the literal $x_i$ is in the clause $B_k$, create the edge $x_ix_{0,k}$. If the literal $\overline{x_i}$ is in the clause $B_k$, create the edge $\overline{x_i}x_{0,k}$. For $i\in [n]$ and $j\in \{0,\ldots,i-1\}$, create the edge $y_{i,j}w$ for every vertex $w\in (X_0\cup\cdots\cup X_{i-1})-X_j$. This ends the reduction of \cite{schaefer-1978}.
See Figure \ref{fig78} for the formula $\Phi$ mentioned above, which is the same example as in \cite{schaefer-1978}.

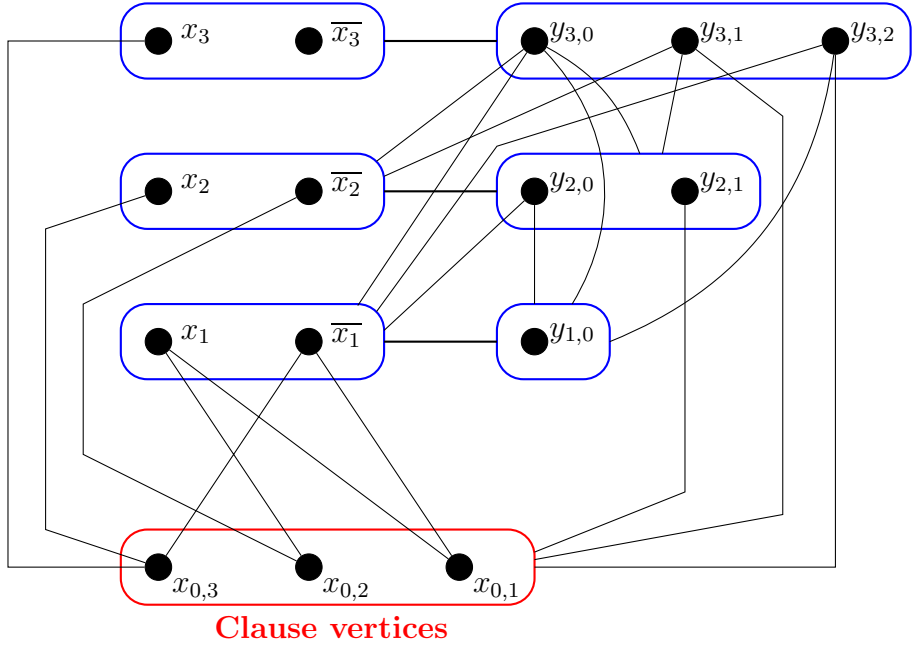
\begin{figure}[ht!]\centering\scalebox{1.0}{
\begin{tikzpicture}
\tikzstyle{vertex}=[draw,circle,fill=black!100,minimum size=10pt,inner sep=.1pt]


\draw[rounded corners=10pt,blue,thick] (-0.5,6.5) rectangle ++(3.5,1);
\node[vertex](x3) at (0,7) {}; \node at (0+0.5,7+0.1) {$x_3$};
\node[vertex](x3b) at (2,7) {}; \node at (2+0.5,7+0.1) {$\overline{x_3}$};
\draw[rounded corners=10pt,blue,thick] (4.5,6.5) rectangle ++(5.5,1);
\node[vertex](y30) at (5,7) {}; \node at (5+0.5,7+0.1) {$y_{3,0}$};
\node[vertex](y31) at (7,7) {}; \node at (7+0.5,7+0.1) {$y_{3,1}$};
\node[vertex](y32) at (9,7) {}; \node at (9+0.5,7+0.1) {$y_{3,2}$};
\draw[thick] (3,7)--(4.5,7);

\draw[rounded corners=10pt,blue,thick] (-0.5,4.5) rectangle ++(3.5,1);
\node[vertex](x2) at (0,5) {}; \node at (0+0.5,5+0.1) {$x_2$};
\node[vertex](x2b) at (2,5) {}; \node at (2+0.5,5+0.1) {$\overline{x_2}$};
\draw[rounded corners=10pt,blue,thick] (4.5,4.5) rectangle ++(3.5,1);
\node[vertex](y20) at (5,5) {}; \node at (5+0.5,5+0.1) {$y_{2,0}$};
\node[vertex](y21) at (7,5) {}; \node at (7+0.5,5+0.1) {$y_{2,1}$};
\draw[thick] (3,5)--(4.5,5);

\draw[rounded corners=10pt,blue,thick] (-0.5,2.5) rectangle ++(3.5,1);
\node[vertex](x1) at (0,3) {}; \node at (0+0.5,3+0.1) {$x_1$};
\node[vertex](x1b) at (2,3) {}; \node at (2+0.5,3+0.1) {$\overline{x_1}$};
\draw[rounded corners=10pt,blue,thick] (4.5,2.5) rectangle ++(1.5,1);
\node[vertex](y10) at (5,3) {}; \node at (5+0.5,3+0.1) {$y_{1,0}$};
\draw[thick] (3,3)--(4.5,3);

\draw[rounded corners=10pt,red,thick] (-0.5,-0.5) rectangle ++(5.5,1);
\node[vertex](x03) at (0,0) {}; \node at (0+0.5,0-0.3) {$x_{0,3}$};
\node[vertex](x02) at (2,0) {}; \node at (2+0.5,0-0.3) {$x_{0,2}$};
\node[vertex](x01) at (4,0) {}; \node at (4+0.5,0-0.3) {$x_{0,1}$};
\node[red] at (2.3,-0.8) {\bf Clause vertices};

\draw (x1)--(x02)(x1)--(x01)(x1b)--(x03)(x1b)--(x01);
\draw (x3)--(-2,7)--(-2,0)--(x03);
\draw (x2)--(-1.5,4.5)--(-1.5,0.5)--(x03);
\draw (x2b)--(-1.0,3.5)--(-1.0,1.5)--(x02);

\draw (y21)--(7,1)--(5,0.2);
\draw (y32)--(9,0)--(5,0);
\draw (y31)--(8.3,6)--(8.3,0.7)--(5,0.1);

\draw (y20)--(3,3.15);
\draw (y30)--(2.65,3.48);
\draw (y32)--(4.5,5.6)--(2.9,3.4);

\draw (y30)--(2.9,5.4); \path[-] (y30) edge[bend left=20] (6.4,5.5);
\draw (y31)--(3,5.2); \draw (y31)--(6.7,5.5);

\draw (y20)--(5,3.5);
\path[-] (y30) edge[bend left=37.5] (5.5,3.5);

\path[-] (y32) edge[bend left=30] (6,3);

\end{tikzpicture}}
\caption{\label{fig78}Example of the construction of Theorem 3.2 of \cite{schaefer-1978} from the input formula $\exists x_3 \forall x_2 \exists x_1: (x_1\vee \overline{x_1})\wedge(x_1\vee\overline{x_2})\wedge(\overline{x_1}\vee x_2\vee x_3)$. A line leading to an enclosure represents multiple edges leading to each vertex in the enclosure. Also, within each enclosure, all vertices are mutually joined by edges, which are not shown.}
\end{figure}

Schaefer \cite{schaefer-1978} defines a game as legitimate if for $i=1,\ldots,n$ the vertex selected at move $i$ is either $x_{n-i+1}$ or $\overline{x_{n-i+1}}$. If a player does an illegitimate move when all previous moves were legitimate, then the other player wins immediately. In order to prove this, fix $i\in[n]$ and assume that the first $n-i$ moves were legitimate, that is, the vertices selected so far are one of each pair $\{x_n,\overline{x_n}\},\ldots,\{x_{i+1},\overline{x_{i+1}}\}$. Clearly, from the construction of $G$, no vertex of $X_n\cup X_{n-1}\cup\ldots\cup X_{i+1}$ can be played now. Suppose that the player on move $n-i+1$ plays illegitimately, by selecting a vertex from $X_i\cup\ldots\cup X_0$ other than $x_i$ or $\overline{x_i}$. If the player selects a node in $X_k$, for $k<i$, then the opponent wins by choosing $y_{i,k}$ and no other vertex is playable, since any vertex of $X_k$ is adjacent to the illegitimately played vertex and any vertex of $(X_0\cup\ldots\cup X_i)-X_k$ is adjacent to $y_{i,k}$. Finally, if the illegitimate move is in $X_i$, it must be $y_{i,k}$ for some $k<i$ and therefore the opponent wins by selecting $x_k$ if $k>0$ or $x_{0,1}$ if $k=0$, which is playable because of the assumption about $B_1$. As before, no other vertex is playable.

As a consequence, we may consider from now on that the players move legitimately.
Then every move on the TQBF game in $\Phi$ has an obvious corresponding move on the Node Kayles game in $G(\Phi)$, and vice-versa: setting $x_i$ true (resp. false) in TQBF corresponds to selecting the vertex $x_i$ (resp. $\overline{x_i}$) in Node Kayles, and vice-versa.

If Player 1 has a winning strategy on a formula $\Phi$ of TQBF, Player 1 can play the corresponding moves on the graph $G(\Phi)$ of Node Kayles. After the last move of TQBF, all clauses are satisfied (because of the winning strategy of Player 1 on TQBF) and then no clause vertex is playable on Node Kayles. Since $n$ is odd, Player 1 is the last to play, winning the Node Kayles game.

Finally, if Player 2 has a winning strategy on a formula $\Phi$ of TQBF, Player 2 can play the corresponding moves on the graph $G(\Phi)$ of Node Kayles. After the last move of TQBF, there is at least one non-satisfied clause (because of the winning strategy of Player 2 on TQBF) and then there is a playable clause vertex on Node Kayles. Since $n$ is odd, Player 1 selected either $x_1$ or $\overline{x_1}$ and Player 2 can select a playable clause vertex in the next move, winning the Node Kayles game.
\end{proof}

\begin{theorem}\label{teo-pspace3}
The \textit{mis\`ere} clique-forming game and the \textit{mis\`ere} Node Kayles game are PSPACE-complete.
\end{theorem}

\begin{proof}
Since the \textit{mis\`ere} clique-forming game is the \textit{mis\`ere} Node Kayles game played on the complement of the graph (and vice-versa), we will only prove that the \textit{mis\`ere} Node Kayles game is PSPACE-hard. We obtain a reduction from the TQBF problem, based on the reduction of \cite{schaefer-1978} for the Node Kayles game (see Theorem \ref{teo-pspace-1978} above).

Following the proof of Theorem \ref{teo-pspace-1978}, instead of assuming that the number $n$ of variables is odd, we assume that $n$ is even. As in the proof of Theorem \ref{teo-pspace-1978}, we can do this since it is possible to add a new variable which does not appear in any clause. We also assume that $B_1=(x_1\vee\overline{x_1})$.

Given a TQBF formula $\Phi$, let $G=G(\Phi)$ be the graph constructed in the proof of Theorem \ref{teo-pspace-1978}. Let $G'=G'(\Phi)$ be the graph obtained from $G$ by replacing every vertex $y_{i,j}$ by two false twin vertices $y'_{i,j}$ and  $y''_{i,j}$ (respecting the adjacency of $y_{i,j}$). That is, $y'_{i,j}$ and  $y''_{i,j}$ are not adjacent in $G'$ and there are edges $y'_{i,j}w$ and $y''_{i,j}w$ in $G'$ for every $w\in V(G)$ such that $y_{i,j}w\in E(G)$.

As in the proof of Theorem \ref{teo-pspace-1978}, we can define a move $i$ as legitimate if the selected vertex is either $x_{n-i+1}$ or $\overline{x_{n-i+1}}$.

We also prove that, if an illegitimate move is done after all previous moves were legitimate, then the other player wins the \emph{mis\`ere} Node Kayles game. For this, fix $i\in[n]$ and assume that the first $n-i$ moves were legitimate. From the construction of $G'$, no vertex of $X_n\cup X_{n-1}\cup\ldots\cup X_{i+1}$ can be played now. Suppose that the player on move $n-i+1$ plays illegitimately, by selecting a vertex from $X_i\cup\ldots\cup X_0$ other than $x_i$ or $\overline{x_i}$. If the player selects a node in $X_k$, for $k<i$, then the opponent wins by choosing $y'_{i,k}$ and then $y''_{i,k}$ is the only playable vertex, since any vertex of $X_k$ is adjacent to the illegitimately played vertex and any vertex of $(X_0\cup\ldots\cup X_i)-X_k$ is adjacent to $y'_{i,k}$. Finally, if the illegitimate move is in $X_i$, it must be either $y'_{i,k}$ or $y''_{i,k}$ for some $k<i$ and therefore the opponent wins by selecting $x_k$ if $k>0$ or $x_{0,1}$ if $k=0$, which is playable because of the assumption about $B_1$. As before, if the illegitimate move was in $y'_{i,k}$ (resp. $y''_{i,k}$), then $y''_{i,k}$ (resp. $y'_{i,k}$) is the only playable vertex.

As a consequence, we may consider from now on that the players move legitimately in the \emph{mis\`ere} Node Kayles game. Then every move on the TQBF game in $\Phi$ has an obvious corresponding move on the \emph{mis\`ere} Node Kayles game in $G'(\Phi)$, and vice-versa: setting $x_i$ true (resp. false) in TQBF corresponds to selecting the vertex $x_i$ (resp. $\overline{x_i}$) in \emph{mis\`ere} Node Kayles, and vice-versa.

If Player 1 has a winning strategy on a formula $\Phi$ of TQBF, Player 1 can play the corresponding moves on the graph $G'(\Phi)$ of Node Kayles. After the last move of TQBF, all clauses are satisfied (because of the winning strategy of Player 1 on TQBF) and then no clause vertex is playable on Node Kayles. Since $n$ is even, Player 2 is the last to play, losing the \emph{mis\`ere} Node Kayles game.

Finally, if Player 2 has a winning strategy on a formula $\Phi$ of TQBF, Player 2 can play the corresponding moves on the graph $G'(\Phi)$ of Node Kayles. After the last move of TQBF, there is at least one non-satisfied clause (because of the winning strategy of Player 2 on TQBF) and then there is a playable clause vertex on \emph{mis\`ere} Node Kayles. Since $n$ is even, Player 2 selected either $x_1$ or $\overline{x_1}$ and Player 1 must select a playable clause vertex in the next move, losing the \emph{mis\`ere} Node Kayles game.
\qed
\end{proof}

\section{The avoidance game played on Cartesian products}
\label{sec:Cartesian}

 The {\em Cartesian product} $G\cp H$ of two graphs $G$ and $H$ is the graph with vertices $V(G)\times V(H)$, where the vertices $(g_1,h_1),(g_2,h_2)$ are adjacent if they are equal in one coordinate and adjacent in the other. If $g\in V(G)$, then the subgraph of $G\cp H$ induced by $\{g\}\times V(H)$ is isomorphic to $H$. It is called the {\em $H$-layer $($through $g)$}, denoted by ${ }^{g}H$. Analogously {\em $G$-layers} $G^h$ are defined. If $S\subseteq V(G\cp H)$, then the set $\{g\in v(G):\ (g,h)\in S\ {\rm for\ some}\ h\in V(H)\}$ is the {\em projection} $\pi_G(S)$ of $S$ on $G$. The projection $\pi_H(S)$ of $S$ on $H$ is defined analogously.
 
We further recall the following known result about the distance function in the Cartesian product. If $G$ and $H$ are connected graphs and $(g,h), (g',h')\in V(G\cp H)$, then the following {\em distance formula} holds: 
\begin{equation}
\label{eq:distance-formula}
d_{G\cp H}((g,h), (g', h')) = d_G(g,g') + d_H(h,h')\,.
\end{equation}
Moreover, if $P$ is a $(g,h),(g',h')$-geodesic in $G\cp H$, then the projection $\pi_G(P)$ induces a $g,g'$-geodesic in $G$ and the projection $\pi_H(P)$ induces a $h,h'$-geodesic in $H$. The distance formula~\eqref{eq:distance-formula} also implies that 
\begin{equation}
\label{eq:intervals-in-CP}
I_{G\cp H}((g,h), (g', h')) = I_G(g,g') \times I_H(h,h')\,.
\end{equation} 
For these results and more information on the Cartesian product we refer to the book on product graphs~\cite{hik-2011}. We make use of the following three known lemmas on general position sets of Cartesian products.

\begin{lemma}
	{\rm\cite[Lemma 2.4]{tian-2021b}}
	\label{lemma:2.4}
Let $G$ and $H$ be connected graphs and let $R$ be a general position set of $G\cp H$. If $u=(g,h)\in R$, then $V(^{g}{H})\cap R=\{u\}$ or $V(G^h)\cap R=\{u\}$. 
\end{lemma}

\begin{lemma}
{\rm\cite[Lemma 3.2]{klavzar-2021e}}
\label{lemma:3.7}
Let $R\subseteq G\cp H$ be such that $R$ has following properties.
\begin{enumerate}
\item Either $V(^{g}{H})\cap R=\{(g,h)\}$ or $V(G^h)\cap R= \{(g,h)\}$, for all $(g,h)\in R$.
\item Both $\pi _G(R)$ and $\pi_H(R)$ are general position sets of $G$ and $H$ respectively. 
\end{enumerate}
Then $R$ is a general position set of $G\cp H$.
\end{lemma}

\begin{lemma}
{\rm\cite[Lemma 3.3]{klavzar-2021e}}
\label{lemma:3.8}
Let $G$ and $H$ be connected graphs and let $R\subseteq V(G\cp H)$. If $\pi _G(R)$ is a general position set in $G$ and $\pi _G(R)=|R|$, then $R$ is a general position set of $G\cp H$.
\end{lemma}
In \cite[Theorem 3.6]{klavzar-2021e} it was proved that player A wins the $\gp$-achievement game on the rook's graphs $K_n\cp K_m$ if and only if both $n$ and $m$ are odd. Thus we first treat the $\gp$-avoidance game on $K_n\cp K_m$. As it happens, the present proof is much more involved than the one for the $\gp$-achievement game, again demonstrating that the two games can be intrinsically different on the same class of graphs. 

\begin{theorem}\label{thm:p-complete}If $n,m\geq 2$, then B wins the $\gp$-avoidance game on $K_n\cp K_m$  if and only if either $n=2$ and $m$ is odd; or $n=3$ and $m$ is even.
\end{theorem}
\proof 
Consider the $\gp$-avoidance game on $K_n\cp K_m$. Let $V(K_n)=\{u_1,\ldots,u_n\}$ and $V(K_m)=\{v_1,\ldots,v_m\}$. Set $G=K_n\cp K_m$ for the rest of the proof. 

We first show that B wins the game on $K_2\cp K_m$ when $m$ is odd. By the vertex transitivity of $G$ we may assume that A start the game with $a_1=(u_1,v_1)$. Then B replies by picking a new vertex from the same $K_m$-layer, say $b_1=(u_1,v_2)$.  Then Lemma~\ref{lemma:2.4} guarantees that both $(u_2,v_1), (u_2, v_2)\notin \Pl_G(a_1, b_1)$. The game then continues in this manner, that is, whenever there exists a $K_2$-layer in which no vertex has yet been played and it is B's turn, he selects a vertex from the $K_m$-layer in which A has just played. Lemmas~\ref{lemma:2.4} and~\ref{lemma:3.7} guarantee that the selected vertices so far form a general position set and the corresponding vertices in the neighboring $K_m$-layer are not playable further. As $m$ is odd, A will play the last vertex and hence B wins.

Next we are going to prove that B wins on $K_3\cp K_m$ if $m$ is even. Let A start the game with $a_1=(u_1,v_1)$. Then B picks the vertex $b_1=(u_2,v_1)$. If A chooses $a_2=(u_3,v_1)$, then it follows from Lemma \ref{lemma:2.4} that no further moves are possible. Hence B wins the game. So we may assume that A pick $a_2$ from a new $K_3$-layer. Also by Lemma \ref{lemma:2.4}, $\Pl_{G}(a_1,b_1)\subseteq $ $^{u_3}K_m$.  By the symmetry of $G$, we may assume that $a_2=(u_3,v_2)$. Now, since $m$ is even and $\Pl_G(a_1,b_1,a_2)= \{u_3\}\times (V(K_m)\backslash \{v_1,v_2\})$, A plays the last vertex. Also Lemma \ref{lemma:3.7} guarantees that the selected vertices so far form a general position set of $G$. Hence B wins the game.\\

It remains to prove that in all the other cases A has a winning strategy. Let A begin by choosing $a_1=(u_1,v_1)$. For the first move $b_1=(u_i,v_j)$ of B we may, using the symmetry of $G$, without loss of generality assume that $i=2$ and $j\in[2]$. Also in the rest of the proof, whenever $(u_r,v_s)$ is selected before $(u_k,v_l)$, then due to the symmetry of $G$ we may assume that $r\leq k$ or $s\leq l$. The strategy of player A is to achieve the following: after each of his moves $a_i$, for some fixed $i$ to be explicitly determined in different cases,  
\begin{equation}
\label{eq:Hamming-graphs}
|\Pl_G(\ldots a_i) \cap V(K_n^{v_k})|\ 
{\rm is\ odd\ and}\ 
|\Pl_G(\ldots a_i) \cap V(^{u_j}K_m)|\
{\rm is\ even} 
\end{equation} hold for all layers $K_n^{v_k}$ and all layers $^{u_j}K_m$ in which at least one vertex has already been played.

\medskip
Suppose first that $b_1=(u_2,v_1)$. Then $\Pl_{G}(a_1, b_1)\subseteq (V(G)\backslash (^{u_1}K_m\cup{ }^{u_2}K_m ))$. We consider the following three cases.

\medskip\noindent
{\bf Case 1.} Both $n$ and $m$ are even. \\
 The strategy of player A is to achieve the condition~\eqref{eq:Hamming-graphs}, after each of his moves $a_i$ with $i\geq 2$. If $n=2$, then ~Lemma \ref{lemma:2.4} implies that $\Pl_G(a_1,b_1)=\emptyset$ and so $A$ wins the game. So, assume that $n\geq 4$ and hence A can continue by
choosing $a_3=(u_3,v_1)$. Again by ~Lemma \ref{lemma:2.4}, $\Pl_{G}(a_1, b_1,a_2)\subseteq V(G)\backslash (^{u_1}K_m\cup{ }^{u_2}K_m \cup {}^{u_3} K_m)$, the condition~\eqref{eq:Hamming-graphs} is fulfilled. In the rest of the game whenever B select a vertex $x$, then A replies with a neighbor of $x$ from the same $K_n$-layer. Thus at any stage  of the game A can make B to choose from a new $K_m$-layer. Since $n$ is even, B will play the last vertex. Also by using ~Lemma \ref{lemma:3.7} we can see that the selected vertices form a general position set of $G$. Hence A will win the $\gp$-avoidance game on $G$. 

\medskip\noindent
{\bf Case 2.} Both $n$ and $m$ are odd. \\
As in the previous case, A chooses $a_2=(u_3,v_1)$ for the second move. After that for each $i=2$ to $\lfloor\frac{n}{2}\rfloor-1$, whenever B selects $b_i=(u_{2i},v_1)$ from $K_n^{v_1}$, A replies with a vertex $a_{i+1}=(u_{2i+1},v_1)$ from $K_n^{v_1}$. This is a legal move since $n$ is odd.  Furthermore,  if B proceeds with $b_{\lfloor\frac{n}{2}\rfloor}=(u_{n-1},v_1)$  from $K_n^{v_1}$, then A deliberately picks the vertex $a_{\lfloor\frac{n}{2}\rfloor+1}=(u_n,v_2)$. Then by  Lemma~\ref{lemma:2.4}, $\Pl_{G}(\ldots a_{\lfloor\frac{n}{2}\rfloor+1})= V (^{u_n} K_m)\setminus \{(u_n,v_1),(u_n,v_2)\}$. Since $m$ is odd, A wins the game. Hence in the following we can assume that for some index $i$ with $2\leq i\leq \lfloor\frac{n}{2}\rfloor$,  B must select $b_i$ from a new $K_n$-layer. Let $r$ be the least such index. Then $a_1,a_2\dots a_{r-1}\in K_n^{v_1}$.  Now, by using Lemma~\ref{lemma:2.4} we can assume that $b_r=(u_{2r},v_2)$. If $r=\lfloor\frac{n}{2}\rfloor$, then A deliberately picks the vertex $a_{r+1}=(u_{2r+1},v_1)=(u_n,v_1)$ for his next move.  Thus again by Lemma~\ref{lemma:2.4}, $\Pl_{G}(\ldots a_{r+1})=  V(^{u_{n-1}} K_m)\setminus \{(u_{n-1},v_1),(u_{n-1},v_2)\}$. Since $m$ is odd, A wins the game. So assume that $2\leq r\leq \lfloor\frac{n}{2}\rfloor-1$. In the remaining game, the strategy of  player A is to achieve the condition~\eqref{eq:Hamming-graphs}, after each of his moves $a_i$ with $i> r$. Then A continues with the vertex $a_{r+1}=(u_{2r},v_3)$ and condition~\eqref{eq:Hamming-graphs} is fulfilled. Hereafter, whenever B selects a vertex in a $K_n$-layer in which at least one vertex has been played earlier, then A can continue with a new vertex from the same $K_n$-layer. Similarly, whenever B selects a vertex from a previously visited $K_m$-layer, then A replies with a new vertex from the same $K_m$-layer. Since $n$ is odd, again condition~\eqref{eq:Hamming-graphs} is fulfilled. On the other hand, consider the case that B plays $b_i=(u_k,v_j)$ from a new layer. In this case if $j$ is odd, then A replies with another such vertex from a new layer; and if $j$ is even, then A replies with a new vertex from the same $K_m$-layer. After each such  move of A, the condition~\eqref{eq:Hamming-graphs} remains fulfilled since both $n$ and $m$ are odd. Note that by our strategy, if $b_i=(u_k,v_j)$, then $k$ is even at each stage of the game. Player A then continues this strategy and wins the game.  

\medskip\noindent
{\bf Case 3.} $n>3$ is odd and $m>2$ is even.  \\
Let player A begin by choosing $a_2=(u_3,v_2)$. After that for each $i=2$ to $\lfloor\frac{n}{2}\rfloor-1$, whenever B selects $b_i=(u_{2i},v_j)$ from $K_n^{v_1}$ or from $K_n^{v_2}$, A replies with the vertex $a_{i+1}=(u_{2i+1},v_j)$ from the same $K_n$-layer. This is a legal move since $n$ is odd.  Furthermore,  if B proceeds with $b_{\lfloor\frac{n}{2}\rfloor}=(u_{n-1},v_j)$  from $K_n^{v_1}$ or from $K_n^{v_2}$,  then A deliberately picks the vertex $a_{\lfloor\frac{n}{2}\rfloor+1}=(u_n,v_3)$. This is a legal move since $m>2$ and it is even. Then by Lemma \ref{lemma:2.4}, either $\Pl_{G}(\ldots a_{\lfloor\frac{n}{2}\rfloor+1})= (V(G)\backslash \{v_1,v_2,v_3\})\times \{u_3,u_n\}$ or  $\Pl_{G}(\ldots a_{\lfloor\frac{n}{2}\rfloor
+1})= (V(G)\backslash \{v_1,v_2,v_3\})\times \{u_n\}$. Since $m$ is even, A wins the game. Hence in the following we can assume that for some index $i$ with $2<i\leq\lfloor\frac{n}{2}\rfloor$, $b_i$ selects from a new $K_n$-layer. Let $r$ be the least such index. Then $a_1,\ldots,a_{r-1}\in V(K_n^{v_1} \cup K_n^{v_1} )$. Now by using Lemma~\ref{lemma:2.4}, we can assume that  $b_r=(u_{2r},v_3)$. If $r=\lfloor\frac{n}{2}\rfloor$, then A picks the vertex $a_{r+1}=(u_{n},v_1)$. Then by  Lemma~\ref{lemma:2.4}, $\Pl_G(\ldots a_{r+1})=(V(G)\backslash \{v_1,v_2,v_3\})\times \{u_{n-1}\}$. Since $m$ is even, A wins the game. So assume that $1<r\leq \lfloor\frac{n}{2}\rfloor-1$. In the remaining game, the strategy of player A is to achieve the condition~\eqref{eq:Hamming-graphs}, after each of his moves $a_i$ with $i> r$. Here A continue with $a_{r+1}=(u_{2r},v_4)$. Since $n>3$ is odd and $m>2$ is even, this is a legal move. By Lemma~\ref{lemma:2.4}, condition~\eqref{eq:Hamming-graphs} is fulfilled. In the rest of the game whenever B picks a vertex in a $K_m$-layer in which at least one vertex has been played earlier, then A chooses a new vertex from the same $K_m$ layer. Similarly, if B picks a vertex in a $K_n$-layer in which at least one vertex has been played earlier, then A chooses a new vertex from the same $K_n$-layer. Now suppose that B plays $b_i=(u_k,v_j)$ from a new layer. In this case if $k$ is odd, A replies with an another such vertex; and if $k$ is even, A selects a vertex from the same $K_m$-layer. Since $m$ is even and $n$ is odd, it follows from Lemma~\ref{lemma:2.4} that the condition~\eqref{eq:Hamming-graphs} remains fulfilled. Note that at any stage of the game, if $b_i=(u_k,v_j)$, then $j$ is odd. Following this strategy, A wins the game.\\

Suppose second that $b_1=(u_2,v_2)$. Again, we need to consider the following three cases.

\medskip\noindent
\textbf{Case 1.} Both $n$ and $m$ are even. \\
In this case the strategy of player A is to achieve the condition~\eqref{eq:Hamming-graphs}, after each of his moves $a_i$ with $i\geq 2$. A continues the game by choosing $a_2=(u_3,v_2)$. Then by ~Lemma \ref{lemma:2.4}, $\Pl_{G}(a_1, b_1,a_2)\subseteq V(G)\backslash (^{u_2}K_m\cup{ }^{u_3}K_m)$. At some point of the game, if B picks a vertex from a $K_n$-layer in which at least one vertex has been played earlier, then A can choose from the same $K_n$-layer. Similarly if B chooses a vertex from a $K_m$-layer in which at least one vertex has been played earlier,  then A can choose a vertex from the same $K_m$-layer. Furthermore, if B plays a vertex from a new layer (no vertex of this layer was played before), then A replies with another such vertex from a new layer.  By Lemma \ref{lemma:3.7}, the selected vertices form a general position set of $G$. Also since both $m$ and $n$ are even the condition~\eqref{eq:Hamming-graphs} remains fulfilled. Hence A wins the game.

\medskip\noindent
{\bf Case 2.} Both $n$ and $m$ are odd. \\
In the remaining game, the strategy of player A is to achieve condition~\eqref{eq:Hamming-graphs}, after each of his moves $a_i$ with $i\geq 2$. A chooses $a_2=(u_2,v_3)$. If B selects a vertex in a $K_n$-layer in which at least one vertex has been played earlier, then choose a vertex from the same $K_n$-layer. And if B plays a vertex such that it is the first vertex played in the two layers, then replies with another such vertex.  And if B chooses a vertex from a $K_m$-layer in which at least one vertex has been played earlier then choose a vertex from the same $K_m$-layer. Since both $n$ and $m$ are odd, the condition~\eqref{eq:Hamming-graphs} remains fulfilled. Hence B plays the last vertex. 

\medskip\noindent
{\bf Case 3.} $n>3$ is odd and $m>2$ is even. \\
A continues by choosing $a_2=(u_3,v_2)$. After that for each $i=2$ to $\lfloor\frac{n}{2}\rfloor-1$, whenever B selects $b_i=(u_{2i},v_j)$ from $K_n^{v_1}$ or $K_n^{v_2}$ , A replies with a vertex $a_{i+1}=(u_{2i+1},v_j)$ from the same $K_n$-layer. This is a legal move since $n$ is odd.  Furthermore,  if B proceeds with $b_{\lfloor\frac{n}{2}\rfloor}=(u_{n-1},v_j)$  from $K_n^{v_1}$ or $K_n^{v_2}$,  then A deliberately picks the vertex $a_{\lfloor\frac{n}{2}\rfloor+1}=(u_n,v_3)$. Since $m>2$ and even, this is a legal move. Then by  Lemma~\ref{lemma:2.4}, $\Pl_{G}(\ldots,a_{\frac{n}{2}+1})= (V(G)\backslash \{v_1,v_2,v_3\})\times \{u_1,u_n\}$. Since $m$ is even and  by  Lemma~\ref{lemma:2.4}, A wins the game. So assume that for some index i B selects a vertex from a new $K_n$-layer. Let $r$ be the least such index. If $r=\lfloor\frac{n}{2}\rfloor$, choose $a_{r+1}=(u_n,v_1)$. Then $\Pl_{G}(\ldots a_{r+1})=(V(G)\backslash \{v_1,v_2,v_3\})\times \{u_{n-1}\}$. Since $m$ is even, A wins the game. Now suppose that   $2\leq r<\lfloor\frac{n}{2}\rfloor$. Then by Lemma~\ref{lemma:2.4}, $b_r=(u_{2r},v_3)$. In the remaining game, the strategy of player A is to achieve the condition~\eqref{eq:Hamming-graphs}, after each of his moves $a_i$ with $i> r$. A continues the game by choosing $a_{r+1}=(u_{2r},v_4)$ . Since $m>2$ this is a legal move. Now if B picks a vertex in a $K_n$-layer in which at least one vertex has been played earlier, then A replies with a vertex from the same $K_n$-layer. Similarly if B selects a vertex in a $K_m$-layer in which at least one vertex has been played earlier, then select a vertex from the same $K_m$-layer. In the rest of the game, suppose that B chooses $b_i=(u_k,v_j)$ from a new layer. If $k$ is odd, then A replies by another such vertex; and if  $k$ is even, A selects a vertex from the same $K_m$-layer. Since $m$ is even and $n$ is odd The condition~\eqref{eq:Hamming-graphs} remains fulfilled. Note that at any stage, by our strategy if $b_i=(u_k,v_j)$, then $j$ is odd. Hence B plays the last vertex. Following this strategy, A wins the game.
\qed

In {\rm\cite{klavzar-2021e}}, it is proved that player B wins the $\gp$-achievement game for any connected bipartite graph.
In the following two theorems, we show that the same behaviour occurs in grids $P_n\cp P_m$, but fails in cylinders $C_n\cp P_m$, which are subclasses of connected bipartite graphs.

\begin{theorem}
\label{theorem:4.4} If $n\geq 3$ and $m\geq 2$, then B wins the $\gp$-avoidance game on $P_n\cp P_m$.
\end{theorem}

\proof Let $P_n = u_1\ldots u_n$, $P_m = v_1\ldots v_m$, and set $G= P_n\cp P_m$. First suppose that $a_1$ is a vertex of degree $2$, say $a_1=(u_1,v_1)$. Then B chooses $b_1=(u_{n-1},v_m)$. Since $I_{G}[a_1,b_1]=V(G)\setminus V(^{u_n}P_m)$, player B forces player A to select a vertex from $V( ^{u_n}P_m)$, which is then also the last vertex played in the game. 

Suppose second that $a_1=(u_i,v_j)$ is a vertex of degree at least 3, say $1<j<m$. Then choose $b_1=(u_1,v_m)$. Then clearly $\Pl_G(a_1,b_1)=\{(u_k,v_l)$ $|$ $1\leq k < i$ and $1\leq l<j\}\cup \{(u_k,v_l)$ $|$ $i< k\leq n$ and $j< l\leq m\}$. Let $a_2=(u_k,v_l)$. Now for any $(u_g,v_h)\in \Pl_{G}(a_1,b_1,a_2)$, either $\{a_1,a_2,(u_g,v_h)\}$ or $\{b_1,a_2,(u_g,v_h)\}$ is not a general position set in $G$. Hence B wins the gp avoidance game on $G$.
\qed

We need the following lemma to prove our next result.

\begin{lemma}{\rm\cite[Lemma 3.1]{klavzar-2021a}}\label{lemma:cylinder}
Let $r\geq 2,s\geq 3$, and $S$ be a general position set of the cylinder graph $P_r\cp C_s$. If $|S\cap V(^{i}C_s)|=2$ for some $i\in[r]_0$, then $|S|\leq 4$.
\end{lemma}
 
\begin{theorem}\label{theorem:3.13b}
If $n\geq 3$ and $m\geq 2$, then B wins the $\gp$-avoidance game on $C_n\cp P_m$ if and only if $n$ is odd.
\end{theorem}

\proof
Let $C_n = u_1\ldots u_{n}u_1$, $P_m = v_1\ldots v_m$, and set $G=C_n\cp P_m$. Consider the $\gp$-avoidance game on $C_n\cp P_m$. 

Suppose first that $n$ is even. We will show that A wins on $G$. Let A start with the vertex $a_1=(u_1,v_1)$.   Let $b_1=(u_i,v_j)$. Since $G$ is connected and bipartite, by using Lemma \ref{lem:bipartite} we can assume that at each stage of the game the selected vertices so far is an independent set of $G$. First if $j=1$, then $2<i< n$. Choose $a_2=(u_2,v_2)$ when $i<\frac{n}{2}$. Since each layer of $G$ is convex, we have that $a_2\notin I_G[a_1,b_1]$. Also since $a_1\notin I_G[a_2,b_1]$ and $b_1\notin I_G[a_1,a_2]$, the second move of A is legal. Now we can see that  $(u_{{\frac{n}{2}+2}},v_2)\in \Pl_{G}(a_1,b_1,a_2)$ since $\{u_i,u_1,u_{\frac{n}{2}+2}\}$ is a general position set of $C_n$ and $u_2\notin I_{C_n}[u_1,u_{\frac{n}{2}+2}]$. Next if $i\geq\frac{n}{2}$, choose $a_2=(u_n,v_2)$. This is a legal move as $I_G[a_1,a_3]=\{(u_n,v_1),(u_1,v_2)\}$. Then using a parallel argument we can see that $(u_2,v_2)\in \Pl_G(a_1,b_1,a_2)$. Hence in both cases, $\Pl_G(a_1,b_1,a_2)\neq \emptyset$. Thus Lemma \ref{lemma:cylinder} in turn implies that A will win the game on $G$. 

Next consider the case $i=1$. Then Lemma \ref{lem:bipartite} in turn implies that $j>2$. Thus A can choose $a_2=(u_{\frac{n}{2}},v_2)$. Since $\{u_1,u_{\frac{n}{2}},u_n\}$ is a general position set of $C_n$, $(u_n,v_2)\in \Pl_G(,\dots,a_2)$. Hence, by Lemma~\ref{lemma:cylinder}, A will win the game on $G$.

For $n$ even, it remains to consider the case $j>1$ and $i>1$  If $i>\frac{n}{2}+1$, then choose $a_3=(u_2,v_j)$. Since $u_2\notin I_{C_n}[u_i,u_1]$, this is a legal move. Then $(u_{i+1},v_1)\in \Pl_G(\dots a_2)$ if $i\neq n$; and $(u_{n-1},v_1)\in \Pl_G(\dots a_2)$ if $i=n$. And if $i\leq\frac{n}{2}$, choose $a_3=(u_n,v_j)$. Then $(u_k,v_1)\in \Pl_G(\dots a_2)$, where $\frac{n}{2}+1<k<n$. Hence as in the previous case A wins the game by Lemma \ref{lemma:cylinder}.

Suppose second that $n$ is odd. In the following we prove that B wins the game on $G$. Let A chooses $a_1=(u_i,v_j)$. Then B replies with the vertex $b_1=(u_{i+1},v_j)$. Let $a_3=(u,v)$. Since both $a_1$ and $b_1$ are adjacent in $G$, the set $\{u_i,u_{i+1},u\}$ must be a general position set of $C_n$. Otherwise, either $a_1\in I_G[a_2,a_3]$ or $a_2\in I_{G}[a_1,a_3]$, which is impossible. Thus $\Pl_{G}(a_1,b_1)\subseteq $ $V(^{u_{i+\lfloor\frac{n}{2}\rfloor}}P_m)$. Let $a_3=(u_{i+\lfloor\frac{n}{2}\rfloor},v_j)$, where $1\leq j \leq m$. Let $x=(u_{i+\lfloor\frac{n}{2}\rfloor},v_k)$ be an arbitrary vertex in $V(^{u_{i+\lfloor\frac{n}{2}\rfloor}}P_m) \setminus \{a_3\}$. Without loss of generality  we may assume that $k>j$. Then $a_3\in I_G[a_2,x]$. Thus $|\Pl_{G}(a_1,b_1)\cap V(^{u_{i+\lfloor\frac{n}{2}\rfloor}}P_m)|=1$. Hence B wins the $\gp$-avoidance game on $G$.
\qed

\section{The avoidance game played on lexicographic products}
\label{sec:Lexico}

The {\em lexicographic product} $G\circ H$ of graphs $G$ and $H$ is the graph with the vertex set $V(G)\times V(H)$, where vertices $(g,h)$ and $(g',h')$ are adjacent if either $gg'\in E(G)$, or $g=g'$ and $hh'\in E(H)$. Layers and projections are defined for the lexicographic product in the same way as they are defined for the Cartesian product. The distance between vertices in lexicographic products can be computed as follows (see \cite{hik-2011}).  

\begin{proposition}
\label{prop:Products:LexDist}
If $(g,h)$ and $(g',h')$ are two vertices of $G\circ H$, then
$$d_{G\circ H}\left((g,h),(g',h')\right)=
\left\{\begin{array}{ll}
d_G(g,g'); & \mbox{$g\ne g'$}\,,\\
d_H(h,h'); & g=g', \deg_G(g)=0\,,\\
\min\{d_H(h,h'), 2\}; & g=g', \deg_G(g)\ne 0\,.\\
\end{array}\right.$$
\end{proposition}

\begin{lemma}{\rm\cite[Lemma 4.2]{klavzar-2021e}}
\label{lemma:3.15}
 Let $G$ and $H$ be connected graphs. If $S$ is a general position set of $G\circ H$, then $\pi_G(S)$ is a general position set of $G$.
\end{lemma}

We need the following lemma which was proved as a claim within the proof of Theorem 4.3 in~\cite{klavzar-2021e}. 
\begin{lemma}
\label{lemma:3.15b}
Let $G$ be a connected graphs, $n\ge 1$, and $S\subseteq V(G\circ K_n)$. Then $S$ is a maximal general position set of $G\circ K_n$ if and only if $S= S_G \times V(K_n)$, where $S_G$ is a maximal general position set of $G$.
\end{lemma}

\begin{theorem}
\label{thm:3.17}
If $G$ is a connected graph and $n\ge 1$, then B wins the $\gp$-avoidance game on $G\circ K_n$ if and only if B wins the $\gp$-avoidance game on $G$ and $n$ is odd.
\end{theorem}

\proof Suppose first that $n$ is odd and B wins the $\gp$-avoidance game on $G$. We will show that B can win the game on $G\circ H$ using the following strategy.  Suppose first that A selects a vertex $a_i$ from a $K_n$-layer from which no vertex was played earlier. If there exists a vertex $x$ such that no vertex from its $K_n$-layer has been played yet and $\pi_G(x)$ is the reply of B to the move $\pi_G(a_{i})$ of A played in $G$ according to the winning strategy of B in $G$, then B selects $b_{i} = x$. Otherwise B replies with an arbitrary vertex from some $K_n$-layer in which at least one vertex has already been played. Suppose second that A selects a vertex $a_i$ from a $K_n$-layer from which at least one vertex was played earlier. Then again B replies with an arbitrary vertex from some $K_n$-layer in which at least one vertex has already been played. Since B wins the $\gp$-avoidance game on $G$, the described strategy ensures that by the end of the game, vertices from an odd number of $K_n$-layers will be played. Moreover, in view of Lemma~\ref{lemma:3.15b}, all the vertices from these layers will be played. Since $n$ is odd, the total number of vertices played by the end of the game will be odd, meaning that A will play the last move and hence B will win the $\gp$-avoidance game on $G\circ K_n$. 

It remains to prove that in the other cases A has a winning strategy.  If $n$ is even, then it follows from Lemma \ref{lemma:3.15b} that an even number of vertices will be played during the game. This shows that A wins the game on $G\circ K_n$.  On the other hand, if A wins on $G$, then A starts the game by choosing a vertex $a_1$ such that $\pi_G(a_1)$ is the starting vertex according to the winning strategy of B in $G$. Then A follows a strategy parallel to the strategy of B from the first paragraph to win the game on $G\circ K_n$. 
\qed

We next give a sufficient condition on $G$ which guarantees that A wins the $\gp$-avoidance game on $G\circ H$. 

\begin{proposition}
\label{prop:3.19}
Let $G$ and $H$ be connected graphs. If there exists a vertex $u$ in $G$ such that for any $v\neq u\in V(G)$, $\Pl_{G}(u,v)\cup\{u,v\}$ is a clique of even order, then A wins the $\gp$-avoidance game on $G\circ H$. 
\end{proposition}

\proof Consider the $\gp$-avoidance game on $G\circ H$. The initial strategy of A is to play $a_1=(u,v_1)$, where $v_1$ is an arbitrary vertex of $H$ and $u$ is such that for any $v\neq u\in V(G)$, $\Pl_{G}(u,v)\cup\{u,v\}$ is a clique of even order in $G$. Then, after the first move of B, Proposition \ref{prop:Products:LexDist} implies that the entire game is restricted to a subgraph of $G\circ H$ induced by $V(G')\times V(H)$, where $G'$ is a clique of $G$ of even order. Moreover, by Lemma~\ref{lemma:3.15b}, all vertices of $V(G')\times V(H)$ will be played by the end of the game. As there are even number of these vertices, B plays the last vertex and A wins. 
\qed

Obvious examples of graphs $G$ that fulfill the assumption of Proposition~\ref{prop:3.19} are complete graphs of even order. Another such family is formed by the graphs that are obtained from a disjoint union of complete graphs of even order by selecting one vertex in each of these complete subgraphs and identifying all of them into a single vertex.

\section{Conclusions and open problems}

To conclude, we provide a list of open questions and directions for further research. The first problem that we propose is the complexity of the $\gp$-achievement game and the $\gp$-avoidance game for graphs with diameter at most $3$. We tried to obtain a reduction from Node Kayles, but the diameter 4 seems to be a strong restriction of our reduction. Are these games polynomial time solvable for graphs with diameter 2? This question is also unsolved.

Another interesting problem is to determine the player with a winning strategy in other simple graph classes, such as cographs and $P_4$-sparse graphs \cite{jamison92}, and strong product of graphs in both general position games.

It is also interesting to see that the $\gp$-achievement game on generalized wheels $W_{n,m}$ is still unsolved. See the technical discussion after Theorem \ref{thm:4.2}, which solved the $\gp$-avoidance game on these graphs. There is a relation between the $\gp$-achievement game on generalized wheels and the game in which three selected vertices cannot induce a $P_3$, which is interesting by itself and seems to be hard even in cycles.

Other open problem is to solve the $\gp$-avoidance game in connected bipartite graphs. In {\rm\cite{klavzar-2021e}}, it is proved that player B always wins the $\gp$-achievement game in this graph class, but Theorem \ref{theorem:3.13b} shows that this is not true even in cylinders (a subclass of connected bipartite graphs).

Finally, an interesting problem is to determine the winner in the $\gp$-avoidance game played on $K_n\circ H$, where $H$ is an arbitrary graph.

\section*{Acknowledgments}

Rudini Sampaio was supported by CAPES [88881.197438/2018-01] and [88881.712024/2022-01], CNPq [311070/2022-1], and FUNCAP [186-155.01.00/2021].
Neethu P. K. acknowledges the Council of Scientific and Industrial Research (CSIR), Govt.\ of India for providing financial assistance in the form of Junior Research Fellowship.
Sandi Klav\v{z}ar acknowledges the financial support from the Slovenian Research Agency (research core funding P1-0297, and projects N1-0095, J1-1693, J1-2452).

\bibliographystyle{plain}


\end{document}